\documentclass{article}

\usepackage{amsfonts}
\usepackage{amsthm}
\usepackage{amsmath}
\usepackage{amssymb}
\usepackage{mathtools} %for prescript
\usepackage{authblk} %for affil
\usepackage{xcolor}
\usepackage{soul} %for strikethrough
\usepackage{mathrsfs}
\usepackage{cite}
\usepackage[top=1in, bottom=1in, left=1in, right=1in]{geometry}
\newcommand{\norm}[1]{\left \lVert#1\right \rVert}
\newtheorem{theorem}{Theorem}[section]
\newtheorem{lemma}{Lemma}[section]

\newtheorem{assumption}{Assumption}[section]
\theoremstyle{definition}
\newtheorem{definition}{Definition}[section]

\theoremstyle{remark}
\newtheorem{remark}{Remark}[section]

\numberwithin{equation}{section}

\begin{document}
	
	%	\begin{frontmatter}
	
	\title{Some results on backward stochastic differential equations of fractional order}
	
	%%this line removes the date, but space is still left for it;
	%if used, remove the \vspace{-1cm}
	\date{}
	
	%this gives the date in the form Mon 30 Jan 2012, 8:57pm;
	%if used, retain the \vspace{-1cm}
	%\date{\shortdayofweekname{\day}{\month}{\year}{ }\mydate\today}
    \author[1]{Nazim I. Mahmudov \thanks{Corresponding author. Email: \texttt{nazim.mahmudov@emu.edu.tr}}}
	\author[1,2]{Arzu Ahmadova\thanks{Email: \texttt{arzu.ahmadova@uni-due.de}}}
	
	%	\ead{nazim.mahmudov@emu.edu.tr}

	\affil[1] {Department of Mathematics, Eastern Mediterranean University, Northern Cyprus, Mersin 10, 99628, Turkey}
	\affil[2] {Faculty of Mathematics, University of Duisburg-Essen, Essen, 45127, Germany}
	
	% Latex won't make the title unless told:
	
	\maketitle
	
	%%to remove the space left for date, use:
	
	\begin{abstract}
		\noindent In this article, we deal with fractional stochastic differential equations, so-called Caputo type fractional backward stochastic differential equations (Caputo fBSDEs, for short), and study the well-posedness of an adapted solution to Caputo fBSDEs of order $\alpha \in (\frac{1}{2},1)$ whose coefficients satisfy a Lipschitz condition. A novelty of the article is that we introduce a new weighted norm in the square integrable measurable function space that is useful for proving a fundamental lemma and its well-posedness. For this class of systems, we then show the coincidence between the notion of stochastic Volterra integral equation and the mild solution.\\
		
		\textit{Keywords:}
		Fractional backward stochastic differential equations, backward stochastic nonlinear Volterra integral equation, well-posedness, adapted process, weighted norm
	\end{abstract}
	%	\end{keyword}  
	
	%	\end{frontmatter}
	\section{Introduction}\label{Sec:intro}
	
   %Fractional stochastic differential equations (FSDEs), generalization of differential equations by the use of fractional and stochastic calculus, are more popular due to their applications in mathematical modelling and finance \cite{oksendal,ito,prato,tien,wang}.  % Most of the results on fractional stochastic dynamical systems are limited to prove existence and uniqueness of mild solutions using fixed point theorem \cite{arzu}-\cite{sakthivel}. 
	
	The study of backward stochastic differential equations (BSDEs) has important applications in stochastic optimal control, stochastic differential games, probabilistic formulae for the solutions of quasilinear partial differential equations, and financial markets. The adapted solution for a linear BSDE which arises as the adjoint process for a stochastic control problem was first investigated by Bismut \cite{bismut} in 1973, then by Bensousssan \cite{bensoussan}, while  Pardoux and Peng \cite{pardoux-peng-1990} first studied the result for the existence and uniqueness of an adapted solution to a continuous general non-linear BSDE  which is a terminal value problem for an It\^{o} type stochastic differential equation under uniform Lipschitz conditions of the following form:
	\begin{align*}
	\begin{cases*}
	\mathrm{d}Y(t)=h(t,Y(t),Z(t))\mathrm{d}t + Z(t)\mathrm{d}W(t), t \in [0,T],\\
	Y(T)=\xi.
	\end{cases*}
	\end{align*}
They established existence and uniqueness of an adapted solution using Bihari's inequality which is the most important generalization of the Gronwall-Bellman inequality.

	 Since then, the theory of BSDE became a powerful tool in many fields such as mathematical finance, optimal control, semi-linear and quasi-linear partial differential equations \cite{hu,rong-1997,rong,tessitore}. Later Peng and Pardoux developed the theory and applications of continuous BSDEs in a series of articles \cite{pardoux-peng-1990,pardoux-1999,peng-1993,peng-1992,peng-1991} under the assumption that coefficients satisfy Lipschitz conditions. Tang and Li \cite{tang-li} then applied the idea of Peng \cite{pardoux-peng-1990} to get the first result on the existence of an adapted solution to a BSDE with Poisson jumps for a fixed terminal time and with Lipschitz conditions. Moreover, Mao \cite{mao} obtained a more general result than that of Pardoux and Peng \cite{pardoux-peng-1990}: he proved existence and uniqueness results under mild assumptions applying Bihari's inequality which was the key tool in the proof.

     Among the many scientific articles on backward stochastic differential equations, we will mention only a few with comparison and relation that motivate this work:  
	 \begin{itemize}
	 	\item  Lin \cite{Lin} considered the following backward stochastic nonlinear Volterra integral equation
	 	\begin{equation*}
	 		X(t)+\int_{t}^{T}f(t,s,X(s),Z(t,s))\mathrm{d}s+\int_{t}^{T}\left[ g(t,s,X(s))+Z(t,s)\right] \mathrm{d}W(s)=X.
	 	\end{equation*}
	 	His aim in \cite{Lin} was to look for a pair $\left\lbrace X(s),Z(t,s)\right\rbrace $ such that this pair is $\left\lbrace \mathcal{F}_{t \vee s}\right\rbrace$-adapted and $Z(t,s)$ is related to $t$. This is one of the common points of our results on linear Caputo fBSDEs with \cite{Lin}. The author also defines $Z(t,s)=\tilde{Z}(t,s)-g(t,s)$, $(t,s)\in \mathcal{D}=\left\lbrace  (t,s)\in \mathbb{R}^{2}_{+}; 0\leq t\leq s \leq T\right\rbrace $ which we have defined for linear Caputo fBSDEs. Another main point of intersection with \cite{Lin,mahmudov,mahmudov2,yong,wang-yong,shi-wang} is to use the well-known extended martingale representation theorem in which we consider extended martingale representation to prove well-posedness of linear Caputo fBSDEs as well.
	 	 \item Mahmudov and Mckibben \cite{mahmudov} studied the existence and uniqueness of the adapted solution to a backward stochastic evolution equation in Hilbert space in the following form:
	 	\begin{align*}
	 		\begin{cases*}
	 			\mathrm{d}y(t)= -[Ay(t)+F(t,y(t),z(t))]\mathrm{d}t-[G(t,y(t))+z(t)]\mathrm{d}w(t),\\
	 			y(T)=\xi,
	 		\end{cases*}
	 	\end{align*}
	 	where $A$ is a linear operator which generates a $C_{0}$-semigroup. The authors also applied Bihari's inequality to show existence and uniqueness under a non-Lipschitz condition. The main idea of this work is to establish a fundamental lemma which plays an important role in the theory of BSDEs. Furthermore, they discussed a stochastic maximum principle for optimal control problems in Hilbert space.  
	 	\item  In comparison with our results Shi et al. studied in \cite{shi} the well-posedness of backward doubly stochastic Volterra integral equations (BDSVIEs) in the terms of introduced M-solutions together with Pontryagin type maximum principle for an optimal control problem of FDSVIEs by virtue of the duality principle.
	 	\item Yong \cite{yong} introduced a backward stochastic Volterra integral equation (BSVIEs) in the following form:
	 	\begin{equation*}
	 		Y(t)=f(t)-\int_{t}^{T}h(t,s,Y(s),Z(t,s), Z(s,t))\mathrm{d}s-\int_{t}^{T}Z(t,s)\mathrm{d}W(s), \quad t\in [0,T].
	 	\end{equation*} 
	 	The interesting features are that $Z(t,s)$ depends on $t$ and the drift depends on both $Z(t,s)$ and  $Z(s,t)$ in general. Moreover, he studied the well-posedness and regularity of adapted M-solutions of BSVIEs in his another article \cite{Yong}. For the reader's convenience, it should be noted that the results in \cite{Yong} differ from what we establish in this article. First, in Theorem 3.7 in \cite{Yong} the assumptions about $f$ and $g$ are different, and second, the Lipschitz coefficients are functions that depend on $t$ and $s$, while we assume here that they are positive constants.
	 	\item Recently, Wang and Yong \cite{wang-yong} considered BSVIEs. By taking a look at special case of BSVIEs, they reduced the given BSVIEs to the following form:
	 	\begin{align}\label{eq:1}
	 		Y(t)=\xi + \int_{t}^{T}g(s,Y(s),Z(t,s))\mathrm{d}s-\int_{t}^{T}Z(t,s)\mathrm{d}W(s), \quad t\in [0,T].
	 	\end{align}
	 	The unknown pair that they are looking for is $\left( Y (\cdot), Z(\cdot,\cdot)\right)$ of stochastic processes. Eq. \eqref{eq:1} is comparable with the integral form of BSDE which takes the following form:	
	 	\begin{equation}\label{eq:2}
	 		Y(t)=\xi + \int_{t}^{T}g(s,Y(s),Z(s))\mathrm{d}s-\int_{t}^{T}Z(s)\mathrm{d}W(s), \quad t\in [0,T].
	 	\end{equation} 
	 	If \eqref{eq:2} has a unique $\mathcal{F}_{t}$-adapted solution $(Y(\cdot), Z(\cdot))$ , then this solution is also an adapted solution of \eqref{eq:1} with $Z(t,s)= Z(s)$ and it turns out that BSVIE can be regarded as an extension of BSDE.	
	 \end{itemize}
	  
%If we imposed the same assumption to our results, we could not apply the same Bihari's inequality they applied since there are a few papers dedicated to studying integral inequalities with weakly singular kernels which are of importance in the theory of singular differential equations and not useful to apply our theory. \cite{medved}. In 1996, Medved proposed a new method to solve such inequalities. It should be stressed out that even by adapting one of these inequalities to our results, it does not allow us to show $x(t)\equiv 0$ since our constant term is not a multiplication of given term, but instead, it is just an addend to the whole right-hand-side. In this case, we could not get a benefit to apply Bihari's type inequality to conclude our desired result.
 
	We study a new fractional analogue of BSDEs which is an untreated topic in the present literature, and so-called \textit{Caputo fractional backward stochastic differential equations} of order $\alpha\in (\frac{1}{2},1)$ on $[0,T]$ as follows:
	\begin{equation}\label{fbsde}
	\begin{cases}
	\left( \prescript{C}{t}{D^{\alpha}_{T}}x\right) (t)=Ax(t)-f(t,x(t),y(t,s))-\left[g(t,s,x(t))+y(t,s) \right] \frac{\mathrm{d}w(t)}{\mathrm{d}t},\\
	x(T)=\xi.
	\end{cases}
	\end{equation}
	where $A$ is a $n\times n$ constant matrix and  $(w(t))_{t\geq 0} $ denote a standard $m$-dimensional  Brownian motion defined on a complete probability space $(\Omega,\mathcal{F}, \mathbb{F}, \mathbb{P})$ with natural filtration $\mathbb{F} \coloneqq \left\lbrace \mathcal{F}_{t}\right\rbrace_{t \geq 0}$ satisfying \textit{usual conditions}. Let $\Omega$ be a sample space, $x(t,\omega)=x(t)$ be a stochastic process on $ [0,T]$ and $\omega\in \Omega$, let $\mathcal{D}=\left\lbrace  (t,s)\in \mathbb{R}^{2}_{+}; 0\leq t\leq s \leq T\right\rbrace $. Moreover, let $f: [0,T]\times \mathbb{R}^{n} \times  \mathbb{R}^{n\times m} \to \mathbb{R}^{n}$ and $g: \mathcal{D}\times \mathbb{R}^{n} \to \mathbb{R}^{n\times m}$ be measurable mappings, and let the terminal value $\xi \in L^{2}(\Omega, \mathcal{F}_{T},\mathbb{P})$ be an $\mathcal{F}_{T}$-measurable $\mathbb{R}^{n}$-valued random variable such that $\textbf{E}\|\xi\|^{2}< \infty$. 
	
	A natural question arises here that why we consider the fractional order $\alpha$ to be in $(\frac{1}{2},1)$. The reason is that in the theory of fractional stochastic differential equations, the fractional order $\alpha\in (\frac{1}{2},1)$ is related to the properties of gamma function. Thus the gamma function is meromorphic on all of $\mathbb{C}$ with simple poles only at the points $0, -1, -2, -3, \ldots$. It is interesting to note that the inverse gamma function $\frac{1}{\Gamma(x)}$ is an entire function: holomorphic for all $x\in \mathbb{C}$. The poles of $\Gamma$ become zeros of this function, and $\Gamma$ has no zeros. Since $\frac{1}{\Gamma(2\alpha-1)}$ appears in our calculations, we must consider the fractional order $\alpha$ as lying in $(\frac{1}{2},1)$ throughout this article.
	
	Our aim in this article is to look for a pair of stochastic processes $\left\lbrace \left( x(t),y(t,s)\right) ; (t,s)\in \mathcal{D}\right\rbrace $ with values $\mathbb{R}^{n}\times \mathbb{R}^{n\times m}$ which we require to be $\mathcal{F}_{t\vee s}$ -adapted and to satisfy \eqref{fbsde} in the usual It\^{o}'s sense. Such a pair is called an adapted solution of the equation \eqref{fbsde}. A novelty of the article is that we introduce a new weighted norm in the square integrable measurable function space that is useful for proving a fundamental lemma and its well-posedness for an adapted pair $\left\lbrace \left( x(t),y(t,s)\right); (t,s)\in \mathcal{D}\right\rbrace $ which solves \eqref{fbsde}. We first derive representations of the adapted solution and then we present that the operator is well-defined before proving a contraction mapping principle.
	
	Therefore, the plan of this article is organized as below: In Section \ref{Sec:prel} we introduce main definitions from stochastic and fractional calculus, and some necessary inequalities from stochastic calculus which play a key role in certain estimations of the main results. Section \ref{stochastic} is devoted to proving global existence and uniqueness  of the adapted solution to Caputo fBSDE \eqref{fbsde} of order  $\alpha \in (\frac{1}{2},1)$ using the Banach's fixed point approach. Section \ref{coincidence} is dedicated to presenting the coincidence between the notion of integral equation and mild solution. Finally, Section \ref{open problems} is for conclusion and introducing some open problems.
	
	\subsection{Mathematical description}
	To conclude the introductory section, we introduce some notations which will be used throughout this article. Let $\mathbb{R}^{n}$ be endowed with the standard norm and $\langle x, \bar{x} \rangle$ denote the inner product of $x, \bar{x} \in \mathbb{R}^{n}$. An element $y\in \mathbb{R}^{n\times m}$ will be considered as a $n\times m$ matrix and its norm is defined by $\|y\|=\sqrt{\text{tr}( y y^{T})}$ and $(x,y)=\text{tr}(xy^{*})$. For any given $0\leq t\leq T$, we denote by $L^{2}_{\mathcal{F}}(0,T;\mathbb{R}^{n})$ (esp. $L^{2}_{\mathcal{F}}(\mathcal{D};\mathbb{R}^{n\times m})$) the family of $\mathbb{R}^{n}$-valued (resp. $\mathbb{R}^{n\times m}$-valued )  $\mathcal{F}_{t}$-adapted processes (resp. $\mathcal{F}_{t\vee s}$)  which are measurable and square integrable on $\Omega\times [0,T]$ with respect to $\mathbb{P} \times \lambda$ where $\lambda$ denotes the Lebesgue measure on $[0,T]$.
	\begin{definition}
		Let $\frac{1}{2}< \alpha < 1$ and $1 \leq p < \infty$ be fixed. We say that a measurable function $h : \Omega\times [0, T] \to \mathbb{R}^{n}$ belongs to $L^{p,\alpha}([0, T], \mathbb{R}^{n})$ if and only if the quantity
		\begin{equation*}
		\|h\|_{p,\alpha,t} \coloneqq  \sup_{t \leq \tau \leq T}\textbf{E}\left( \int_{\tau}^{T}\frac{\|h(s)\|^{p}}{(s-\tau)^{1-\alpha}}\mathrm{d}s\right)^{\frac{1}{p}} < \infty.
		\end{equation*} 
		\begin{lemma}
		Let $\frac{1}{2}< \alpha < 1$ and $1 \leq p < \infty$, $L^{p,\alpha}$ is a Banach space.	
		\end{lemma}
	\begin{proof}
		Obviously, the defined weighted norm satisfies the conditions of a norm in the  $L^{p,\alpha}$ space. It is only left to show that it is complete. Hence, a proof of this part is similar to the proof of Theorem 3.4 in \cite{omid}, so we omit it here. 
	\end{proof}
		In particular, we consider $p=2$ and $\frac{1}{2}<\alpha <1$  throughout this article such that \textit{the new weighted norm} is defined by
		\begin{equation}\label{fnorm}
		\|h\|_{2,\alpha,t} \coloneqq \sup_{t \leq \tau \leq T}\textbf{E}\left( \int_{\tau}^{T}\frac{\|h(s)\|^{2}}{(s-\tau)^{1-\alpha}}\mathrm{d}s\right)^{\frac{1}{2}}.
		\end{equation} 	
		To show existence and uniqueness, we consider the following norm:
		\begin{equation}\label{hnorm0}
		\|h\|_{2,\alpha,0} \coloneqq \sup_{0 \leq \tau \leq T}\textbf{E}\left( \int_{\tau}^{T}\frac{\|h(s)\|^{2}}{(s-\tau)^{1-\alpha}}\mathrm{d}s\right)^{\frac{1}{2}}.
		\end{equation}
	\end{definition}
	
	%\mathbb{L}^{2}_{\mathcal{F}}(t,T;\mathbb{R}^{n})\times \mathbb{L}^{2}(t,T;\mathbb{L}_{\mathcal{F}}(t,T;\mathbb{R}^{n\times m}))
	We define  $M[t,T]\coloneqq L^{2}_{\mathcal{F}}(t,T;\mathbb{R}^{n})\times L^{2}_{\mathcal{F}}(\mathcal{D};\mathbb{R}^{n\times m}) $ to be a Banach space endowed with the norm 
	\begin{equation*}
	\|(x,y)\|^{2}_{t}=\|x\|^{2}_{2,\alpha,t}+|\|y\||^{2}_{2,\alpha,t},
	\end{equation*}
	where 
	\begin{align}\label{yt}
	|\|y\||^{2}_{2,\alpha,t}\coloneqq \sup_{t \leq \tau \leq T}\textbf{E}\int_{\tau}^{T}(s-\tau)^{\alpha-1}\int_{s}^{T}\|y(s,u)\|^{2}\mathrm{d}u\mathrm{d}s,\quad  t\leq s\leq u\leq T.
	\end{align}
	Then $M[0,T]\coloneqq L^{2}_{\mathcal{F}}(0,T;\mathbb{R}^{n})\times L^{2}_{\mathcal{F}}([0,T]^{2};\mathbb{R}^{n\times m})$ is also a Banach space equipped with the norm as below:
	\begin{equation*}
	\|(x,y)\|^{2}_{0}=\|x\|^{2}_{2,\alpha,0}+|\|y\||^{2}_{2,\alpha,0}.
	\end{equation*}
	%where 
	%\begin{align}\label{y0}
	%\|y\|^{2}_{0}\coloneqq\sup_{0 \leq \tau \leq T}\textbf{E}\int_{\tau}^{T}(s-\tau)^{\alpha-1}\int_{s}^{T}\|y(t,s)\|^{2}\mathrm{d}s,\quad  0\leq t\leq s\leq T.
	%\end{align}
	%Let $\mathbb{L}^{2}(0,T;\mathbb{R}^{n})$ denote the space of all $\mathscr{F}_{T}$-measurable processes satisfying
	%\begin{equation*}
	%\|y\|^{2}\coloneqq \sup_{t\leq u \le T}\|y(t,u)\|^{2} < \infty.
	%\end{equation*}

\section{Fractional and Stochastic Calculus}\label{Sec:prel}
	We embark on this section by briefly introducing the essential structure of fractional calculus and fractional operators, as well as some important 
	inequalities of stochastic calculus. For the more salient details on these matters, see the textbooks \cite{miller-ross,kilbas,gorenflo,oldham-spanier,samko}. We introduce first the classical fractional operators in this section. Since we study backward stochastic fractional differential equation throughout this article, we consider the \textit{right} stochastic Riemann-Liouville and Caputo fractional operators of order $\alpha>0$ on $[t,T]$.

	\begin{itemize}
		\item \textit{The right Riemann-Liouville stochastic fractional derivative} of order $\alpha$ is as follows:
		\begin{align*}
	    \prescript{}{t}D^{\alpha}_{T}x(t)=\frac{(-1)^{n}}{\Gamma(n-\alpha)}\left(\frac{d}{dt} \right)^{n}\int_{t}^{T}(s-t)^{n-\alpha-1}x(s)\mathrm{d}s, \quad \text{where}\quad n-1<\alpha <  n, \quad t<T.
		\end{align*}
		\item \textit{The right Riemann-Liouville stochastic fractional integral} of order $\alpha$ is given by
		\begin{equation}
		\prescript{}{t}I^{\alpha}_{T}x(t)=\frac{1}{\Gamma(\alpha)}\int_t^T(s-t)^{\alpha-1}x(s)\,\mathrm{d}s \\,\quad \text{for} \quad t<T,
		\end{equation}
		where $\Gamma$ is the gamma function.
		\item \textit{The right Caputo stochastic fractional derivative operator} of order $\alpha$  is defined by:
		\begin{equation*}
		\prescript{C}{t}D^{\alpha}_{T}x(t)=\frac{(-1)^{n}}{\Gamma(n-\alpha)}\int_{t}^{T}(s-t)^{n-\alpha-1}x^{(n)}(s)\mathrm{d}s, \quad \text{where}\quad n-1<\alpha < n, \quad t<T.
		\end{equation*}
		In particular, for $\alpha \in (0,1)$
		\begin{equation*} 
		\prescript{}{t}I^{\alpha}_{T}\prescript{C}{t}D^{\alpha}_{T}x(t)=x(t)-x(T).
		\end{equation*}
		
		\item \textit{The relationship between the Riemann--Liouville and Caputo fractional derivatives} are as follows:
		\begin{equation} \label{eq:relation}
		\prescript{C}{t}D^{\alpha}_{T}x(t)=\prescript{}{t}D^{\alpha}_{T}x(t)-\sum_{k=0}^{n-1}\frac{(T-t)^{k-\alpha}x^{(k)}(T)}{\Gamma(k-\alpha+1)}.
		\end{equation}
		\item \textit{The property of the Riemann--Liouville fractional integral operator and the Caputo fractional derivative}  of order $\alpha$ is given by:
		\begin{equation} \label{ID} \prescript{}{t}I^{\alpha}_{T}(\prescript{C}{t}D^{\alpha}_{T}x(t))=x(t)-\sum_{k=0}^{n-1}\frac{(T-t)^{k}x^{(k)}(T)}{\Gamma(k+1)}.
		\end{equation}
	\end{itemize}  

	The following inequalities are basic inequalities of stochastic calculus (cf. \cite{oksendal,ito,prato}). First, we introduce Jensen's inequality in probability setting which plays an important role in the theory of stochastic processes. It has become a standard result, appearing in almost every introductory text in the field.
	%\begin{theorem}[Doob’s martingale maximal inequality]
		%Let $\left\lbrace \mathscr{F}_t\right\rbrace _{t \ge 0}$ be a filtration on probability space $(\Omega, \mathcal{F},\mathbb{P})$ and let $(M_t)_{t \ge 0}$ be a continuous martingale with respect to the filtration $\left\lbrace \mathcal{F}_t\right\rbrace _{t \ge 0}$.
		%Let $p > 1$ and $T>0$. If $\textbf{E}(\| M_T\|^p)<+\infty$, then we have
		%\begin{equation}\label{doob}
		%\textbf{E}\left( \left( \sup_{0 \le t \le T} \| M_t \| \right)^p \right) \leq \left( \frac{p}{p-1} \right)^p \textbf{E} (\|M_T\|^p).
		%\end{equation}
%	\end{theorem}
\begin{description}
	\item [Jensen's inequality in probabilistic setting:]
		Let $X$ be an integrable real-valued random variable and $\varphi$ be a convex function. Then:
	\begin{equation}\label{jensen}
		\varphi \left(\textbf{E}[X]\right)\leq \textbf{E} \left[\varphi (X)\right].
	\end{equation}
	In this probability setting, $\varphi$ is intended as the integral with respect to an expected value $\textbf{E}$. 
	\item[The law of total expectation:]
		Let $(\Omega,\mathcal{F}_{t},\mathbb{P})$ be a probability space on which two sub $\sigma$-algebras $\mathcal{G}_{1}\subset \mathcal{G}_{2} \subset \mathcal{F}$ are defined. For a random variable $X$ on such a space, the smoothing law states that if $\textbf{E}[X]$ is defined, then
	\begin{equation}
		\textbf{E}[\textbf{E}[X\mid \mathcal{G}_{2}]\mid \mathcal{G}_{1} ]=\textbf{E}[X\mid \mathcal{G}_{1}], \quad \text{a.s.}
	\end{equation}
	In special case, when $\mathcal{G}_1 = \{\emptyset,\Omega \}$ and $\mathcal{G}_{2}=\sigma (Y)$, the smoothing law reduces to 
	\begin{equation}\label{totalaw}
		\textbf{E}[\textbf{E}[X\mid Y]]=\textbf{E}[X].
	\end{equation}
\end{description}

\section{Main results for fractional backward stochastic differential equations}\label{stochastic}

In this section, we study existence and uniqueness results to a mild solution of the equation \eqref{fbsde}. To state our main results, we review the following assumptions which will be considered in Section \ref{stochastic} and \ref{coincidence}. More precisely, let us propose the standing assumptions for the function  $f$ and $g$ as follows.

\begin{assumption}\label{A0}
	For all $x,\bar{x} \in \mathbb{R}^{n}$, $y,\bar{y}\in \mathbb{R}^{n\times m}$ and $0\leq t\leq T$, there exists $c>0$ such that
	\begin{align*}
	&\|f(t,x,y)-f(t,\bar{x},\bar{y})\|^{2}\leq c \left(\|x-\bar{x}\|^{2}+\|y-\bar{y}\|^{2} \right), \quad \text{a.s},\\
	&\|g(t,s,x)-g(t,s,\bar{x})\|^{2}\leq c\|x-\bar{x}\|^{2}\quad \text{a.s}.
	\end{align*}
\end{assumption}
\begin{assumption}\label{A1}
	$f(\cdot,0,0)$ is $L^{2}(0,T;\mathbb{R}^{n})$ integrable i.e.,
	\begin{equation*}
	\int_{0}^{T} \|f(r,0,0)\|^{2}\mathrm{d}r <\infty,
	\end{equation*}
	and $g(\cdot,\cdot,0)\in L^{2}(\mathcal{D}, \mathbb{R}^{n\times m})$ is essentially bounded i.e.,
	\begin{equation*}
	ess\sup\limits_{r,s\in [0,T]}\|g(r,s,0)\|< \infty.
	\end{equation*} 
\end{assumption}

\begin{definition}
	A stochastic process $ \left\lbrace \left( x(t), y(t,s)\right) ,(t,s)\in \mathcal{D}\right\rbrace$ is called  a mild solution of \eqref{fbsde} if 
	\begin{itemize}
		\item $ \left( x,y\right)$ is adapted to $\left\lbrace \mathcal{F}_{t}\right\rbrace _{t \geq 0}$ with 
		$ \int_{0}^{T}  \|(x,y)\|^{2}\mathrm{d}t< \infty$ almost everywhere;
		\item $ (x,y)\in M[t,T]$ has continuous path on $t\in [0,T]$ a.s. and satisfies the following Volterra integral equation of second kind on $t\in [0,T]$:
	\end{itemize}	
\end{definition}	
\begin{align}  \label{integral equation}
x(t)&=\xi-\frac{1}{\Gamma(\alpha)}\int_{t}^{T}A(s-t)^{\alpha-1}x(s)\mathrm{d}s\nonumber\\
&+\frac{1}{\Gamma(\alpha)}\int_{t}^{T}(s-t)^{\alpha-1}f(s,x(s),y(t,s))\mathrm{d}s\nonumber\\
&+\frac{1}{\Gamma(\alpha)}\int_{t}^{T}(s-t)^{\alpha-1}\left[g(t,s, x(s))+y(t,s) \right] \mathrm{d}w(s).
\end{align}
\begin{definition}
	A pair of adapted process $(x,y)\in M[t,T]$ is a mild solution of \eqref{fbsde} for all $t \in [0,T]$ if it satisfies the following backward stochastic nonlinear Volterra integral equation 
	\begin{align}\label{2}
	x(t)=E_{\alpha}(A(T-t)^{\alpha})\xi &+\int_{t}^{T}(s-t)^{\alpha-1}E_{\alpha,\alpha}(A(s-t)^{\alpha})f(s,x(s),y(t,s))\mathrm{d}s\nonumber\\
	&+\int_{t}^{T}(s-t)^{\alpha-1}E_{\alpha,\alpha}(A(s-t)^{\alpha})\left[g(t,s,x(s))+y(t,s) \right] \mathrm{d}w(s), \qquad \mathbb{P}\text{-a.s},
	\end{align}
where $E_{\alpha}(\cdot)$ and $E_{\alpha,\alpha}(\cdot)$ are one and two-parameter Mittag-Leffler functions, respectively, which we give their definitions as follows:
\end{definition}
	\begin{equation}
		E_{\alpha}(t)= \sum_{k=0}^{\infty}\frac{t^{k}}{\Gamma(k \alpha +1)}, \quad  \alpha \in \mathbb{R}_{+}, t\in\mathbb{R},
	\end{equation}
	\begin{equation}
		E_{\alpha,\alpha}(t)= \sum_{k=0}^{\infty}\frac{t^{k}}{\Gamma(k \alpha +\alpha)}, \quad  \alpha \in \mathbb{R}_{+}, t\in\mathbb{R}.
	\end{equation}
These series are convergent, locally uniformly in $t$, provided the $\alpha>0$ condition is satisfied.

We now introduce a fundamental lemma which plays an important role to state and prove existence and uniqueness results.
\begin{lemma}\label{lem1}
For any $f\in L^{2}([0,T], \mathbb{R})$ and $g\in L^{2}(\mathcal{D}, \mathbb{R}^{n})$, the following linear Caputo fBSDE equation:
\begin{align}\label{3}
x(t)=E_{\alpha}(A(T-t)^{\alpha})\xi &+\int_{t}^{T}(s-t)^{\alpha-1}E_{\alpha,\alpha}(A(s-t)^{\alpha})f(s)\mathrm{d}s\nonumber\\
&+\int_{t}^{T}(s-t)^{\alpha-1}E_{\alpha,\alpha}(A(s-t)^{\alpha})\left[g(t,s)+y(t,s) \right] \mathrm{d}w(s), \quad \mathbb{P}\text{-a.s.}
\end{align}
has a unique solution in $M[0,T]$, and moreover we have
\begin{align}\label{ineq1}
\|x\|^{2}_{2,\alpha,t}+|\|y\||^{2}_{2,\alpha,t}&\leq 2M^{2}_{\alpha}\textbf{E}\|\left\lbrace \xi \mid \mathcal{F}_{\cdot} \right\rbrace\|^{2}_{2,\alpha,t}+ 16M^{2}_{\alpha}\frac{T^{\alpha}}{\alpha}\textbf{E}\|\xi\|^{2}\nonumber\\
&+2M^{2}_{\alpha,\alpha}\left(\frac{T^{2\alpha}}{\alpha^{2}}+ \frac{4(2T)^{2\alpha}}{\alpha(2\alpha-1)} \right) \|f\|^{2}_{2,\alpha,t}+2|\|g\||^{2}_{2,\alpha,t},
\end{align}
where $M_{\alpha}$ and $M_{\alpha,\alpha}$ are defined as below since Mittag-Leffler functions are entire and bounded on $[0,T]$:
\begin{align}\centering\label{Mm}
	M_{\alpha}&\coloneqq\sup \left\lbrace \|E_{\alpha}(At^{\alpha})\|, 0\leq t \leq T\right\rbrace,\nonumber\\
	M_{\alpha,\alpha}&\coloneqq\sup \left\lbrace \|E_{\alpha,\alpha}(At^{\alpha})\|, 0\leq t \leq T\right\rbrace.
\end{align}
\end{lemma}
\begin{proof}
	
\textit{Uniqueness:} Let $(x_{1},y_{1})$ and $(x_{2},y_{2})$ be two solutions of \eqref{3}.
\begin{align}\label{xy}
x_{1}(t)-x_{2}(t)=\int_{t}^{T}(s-t)^{\alpha-1}E_{\alpha,\alpha}(A(s-t)^{\alpha})\left[y_{1}(t,s)-y_{2}(t,s) \right] \mathrm{d}w(s),
\end{align}
 Taking $\textbf{E}\left\lbrace \cdot \mid \mathcal{F}_{t} \right\rbrace $ from above, we can deduce for all $t\in [0,T]$ that
 \begin{align}\label{x1x2}
 \textbf{E}\left\lbrace x_{1}(t)-x_{2}(t)\mid \mathcal{F}_{t} \right\rbrace =0.
 \end{align}
By \eqref{x1x2} we obtain $x_{1}(t)=x_{2}(t)$ and substituting this into \eqref{xy} follows that $y_{1}(t,s)=y_{2}(t,s)$, $(t,s)\in \mathcal{D}$.

\textit{Existence:} Taking a conditional expectation from \eqref{3}, we have that
\begin{equation*}
x(t)=E_{\alpha}(A(T-t)^{\alpha})\textbf{E}\left\lbrace \xi \mid \mathcal{F}_{t}\right\rbrace+\int_{t}^{T}(s-t)^{\alpha-1}E_{\alpha,\alpha}(A(s-t)^{\alpha})\textbf{E}\left\lbrace f(s) \mid \mathcal{F}_{t}\right\rbrace\mathrm{d}s.
\end{equation*}
From extended martingale representation theorem, there exists $L(\cdot)\in L^{2}_{\mathcal{F}}(0,T;L^{2}(\mathbb{R}^{n}))$ and uniquely $K(t,\cdot)\in  L^{2}_{\mathcal{F}}(\mathcal{D};L^{2}(\mathbb{R}^{n\times m}))$ which satisfy the following relations:
\begin{equation}\label{L}
\textbf{E}\left\lbrace \xi \mid \mathcal{F}_{t}\right\rbrace=\textbf{E}\xi+\int_{0}^{t}L(u)\mathrm{d}w(u),
\end{equation}
\begin{equation}\label{K}
\textbf{E}\left\lbrace f(s) \mid  \mathcal{F}_{t}\right\rbrace=\textbf{E}f(s)+\int_{0}^{t}K(s,u)\mathrm{d}w(u).
\end{equation}
Note also from \eqref{K}, we can easily deduce for all  $s \in [0,T]$ that 
\begin{equation*}
K(s,u)=0, \quad \text{a.e.}, \quad u\in [s,T], \quad \text{a.s.}
\end{equation*}
 and that satisfies
\begin{equation}
\textbf{E}\int_{0}^{T}\int_{0}^{s}|K(s,u)|^{2}\mathrm{d}u\mathrm{d}s\leq 4\textbf{E}\int_{0}^{T}|f(s)|^{2}\mathrm{d}s.
\end{equation}
Since $t\in [0.T]$, it is obvious that
\begin{align*}
\xi&=\textbf{E}\xi +\int_{0}^{T}L(u)\mathrm{d}w(u)\\
&=\textbf{E}\xi +\int_{0}^{t}L(u)\mathrm{d}w(u)+\int_{t}^{T}L(u)\mathrm{d}w(u)\\
&=\textbf{E}\left\lbrace \xi \mid \mathcal{F}_{t}\right\rbrace +\int_{t}^{T}L(u)\mathrm{d}w(u),
\end{align*}
and since $s\geq t$, we have 
\begin{align*}
f(s)&=\textbf{E}f(s)+\int_{0}^{s}K(s,u)\mathrm{d}w(u)\\
&=\textbf{E}f(s)+\int_{0}^{t}K(s,u)\mathrm{d}w(u)+\int_{t}^{s}K(s,u)\mathrm{d}w(u)\\
&=\textbf{E}\left\lbrace f(s) \mid  \mathcal{F}_{t}\right\rbrace+\int_{t}^{s}K(s,u)\mathrm{d}w(u).
\end{align*}
Therefore, we obtain
\begin{equation}\label{ll}
\textbf{E}\left\lbrace \xi \mid  \mathcal{F}_{t}\right\rbrace=\xi-\int_{t}^{T}L(u)\mathrm{d}w(u),
\end{equation}
and 
\begin{equation}\label{kk}
\textbf{E}\left\lbrace f(s) \mid  \mathcal{F}_{t}\right\rbrace= f(s)-\int_{t}^{s}K(s,u)\mathrm{d}w(u).
\end{equation}
Substituting \eqref{ll} and \eqref{kk} into \eqref{4} and using stochastic Fubini's theorem give that
\allowdisplaybreaks
\begin{align*}
x(t)&=E_{\alpha}(A(T-t)^{\alpha})\left(\xi-\int_{t}^{T}L(u)\mathrm{d}w(u) \right)+\int_{t}^{T}(s-t)^{\alpha-1} E_{\alpha,\alpha}(A(s-t)^{\alpha})\left(f(s)-\int_{t}^{s}K(s,u)\mathrm{d}w(u) \right)\mathrm{d}s \\
&=E_{\alpha}(A(T-t)^{\alpha})\xi +\int_{t}^{T}(s-t)^{\alpha-1} E_{\alpha,\alpha}(A(s-t)^{\alpha})f(s)\mathrm{d}s\\
&-E_{\alpha}(A(T-t)^{\alpha})\int_{t}^{T}L(u)\mathrm{d}w(u)-\int_{t}^{T}(s-t)^{\alpha-1}E_{\alpha,\alpha}(A(s-t)^{\alpha})\int_{t}^{s}K(s,u)\mathrm{d}w(u)\mathrm{d}s\\
&=E_{\alpha}(A(T-t)^{\alpha})\xi +\int_{t}^{T}(s-t)^{\alpha-1} E_{\alpha,\alpha}(A(s-t)^{\alpha})f(s)\mathrm{d}s\\
&-E_{\alpha}(A(T-t)^{\alpha})\int_{t}^{T}L(u)\mathrm{d}w(u)-\int_{t}^{T}\int_{u}^{T}(s-t)^{\alpha-1}E_{\alpha,\alpha}(A(s-t)^{\alpha})K(s,u)\mathrm{d}s\mathrm{d}w(u).
\end{align*}
Thus, we get that
\begin{align*}
x(t)=E_{\alpha}(A(T-t)^{\alpha})\xi +\int_{t}^{T}(s-t)^{\alpha-1} E_{\alpha,\alpha}(A(s-t)^{\alpha})f(s)\mathrm{d}s+\int_{t}^{T}\tilde{y}(t,u)\mathrm{d}w(u).
\end{align*}
Then there exists a mild solution $(x,y)\in M[0,T]$ of \eqref{3} given by
	 \begin{equation}\label{4}
	 x(t)=E_{\alpha}(A(T-t)^{\alpha})\textbf{E}\left\lbrace \xi \mid \mathcal{F}_{t}\right\rbrace+\int_{t}^{T}(s-t)^{\alpha-1}E_{\alpha,\alpha}(A(s-t)^{\alpha})\textbf{E}\left\lbrace f(s) \mid \mathcal{F}_{t}\right\rbrace\mathrm{d}s,
	 \end{equation}
	 and 
	\begin{equation}\label{5}
	\tilde{y}(t,u)=-E_{\alpha}(A(T-t)^{\alpha})L(u)-\int_{u}^{T}(s-t)^{\alpha-1}E_{\alpha,\alpha}(A(s-t)^{\alpha})K(s,u)\mathrm{d}s.
	\end{equation} 
	We finally define $y(t,u)=\tilde{y}(t,u)-g(t,u)$, $(t,u)\in \mathcal{D}$.
It is easily seen that the pair $(x,y)$ solves \eqref{3}. Therefore, the existence is proved.

From \eqref{ll} and \eqref{kk}, we invoke the following inequalities for $0\leq t\leq s\leq T$:
\begin{equation*}
\textbf{E}\int_{t}^{T}\|L(u)\|^{2}\mathrm{d}u\leq 4\textbf{E}\|\xi\|^{2},
\end{equation*}
and
\begin{equation*}
\textbf{E}\int_{t}^{s}\|K(s,u)\|^{2}\mathrm{d}u\leq 4\textbf{E}\|f(s)\|^{2}.
\end{equation*}	 
Now we estimate the solution $(x,y)$ given by  \eqref{4} and \eqref{5} in $[0,T]$. From \eqref{4} it follows that
	 \begin{align}\label{xx}
	 \sup_{t\leq \tau\leq T}\textbf{E}\int_{\tau}^{T}(s-\tau)^{\alpha-1}\|x(s)\|^{2}\mathrm{d}s&\leq 2M^{2}_{\alpha}\sup_{t\leq \tau\leq T}\textbf{E}\int_{\tau}^{T}(s-\tau)^{\alpha-1}\|\textbf{E}\left\lbrace \xi | \mathcal{F}_{s} \right\rbrace\|^{2}\mathrm{d}s\nonumber\\
	 &+2M^{2}_{\alpha,\alpha}\sup_{t\leq \tau\leq T}\textbf{E}\int_{\tau}^{T}(s-\tau)^{\alpha-1}\Big(\int_{s}^{T} (r-s)^{\alpha-1}\textbf{E}\left\lbrace \|f(r)\| \mid \mathcal{F}_{s}\right\rbrace \mathrm{d}r\Big)^{2}\mathrm{d}s\nonumber\\
	 &\coloneqq \mathcal{I}_{1}+\mathcal{I}_{2}.
	 \end{align}
	 By the inequality \eqref{totalaw}, $\mathcal{I}_{1}$ becomes as follows:
	 \begin{align}\label{I1}
	 \mathcal{I}_{1}&\coloneqq 2M^{2}_{\alpha}\sup_{t\leq \tau\leq T}\textbf{E}\int_{\tau}^{T}(s-\tau)^{\alpha-1}\|\textbf{E}\left\lbrace \xi | \mathcal{F}_{s} \right\rbrace\|^{2}\mathrm{d}s=2M^{2}_{\alpha}\textbf{E}\|\left\lbrace \xi \mid \mathcal{F}_{\cdot} \right\rbrace\|^{2}_{2,\alpha,t}.
	 %\sup_{t\leq \tau\leq T}\textbf{E}\int_{\tau}^{T}(s-\tau)^{\alpha-1}\|\textbf{E}\left\lbrace \|\xi\|^{2}| \mathcal{F}_{s} \right\rbrace\|\mathrm{d}s\\
	 %&\leq  8M^{2}_{\alpha}\textbf{E}\left( \textbf{E}\left\lbrace \int_{\tau}^{T}(s-\tau)^{\alpha-1}\|\xi\|^{2}\mathrm{d}s\mid \mathscr{F}_{s} \right\rbrace\right)\\
	 %&\leq 8M^{2}_{\alpha}\textbf{E}\left(  \int_{\tau}^{T}(s-\tau)^{\alpha-1}\|\xi\|^{2}\mathrm{d}s\right)\\
	 %&\leq 8M^{2}_{\alpha}\frac{(T-t)^{\alpha}}{\alpha}\textbf{E}\|\xi\|^{2}
	 \end{align}
	 By the inequality \eqref{jensen}, we obtain that
	 \begin{align}\label{I2}
	 \mathcal{I}_{2}&\coloneqq 2M^{2}_{\alpha,\alpha}\sup_{t\leq \tau\leq T}\textbf{E}\int_{\tau}^{T}(s-\tau)^{\alpha-1}\left( \int_{s}^{T}(r-s)^{\alpha-1}\textbf{E}\left\lbrace \|f(r)\| \mid \mathcal{F}_{s}\right\rbrace \mathrm{d}r  \right)^{2}\mathrm{d}s\nonumber\\
	 &= 2M^{2}_{\alpha,\alpha}\sup_{t\leq \tau\leq T}\textbf{E}\int_{\tau}^{T}(s-\tau)^{\alpha-1}\left(\textbf{E}\Biggl\{ \int_{s}^{T}(r-s)^{\alpha-1}\|f(r)\|\mathrm{d}r \mid \mathcal{F}_{s} \Biggr\} \right)^{2}\mathrm{d}s\nonumber\\
	 &\leq 2M^{2}_{\alpha,\alpha}\sup_{t \leq s \leq T}\int_{\tau}^{T}(s-\tau)^{\alpha-1}\textbf{E}\left( \int_{s}^{T}(r-s)^{\alpha-1}\|f(r)\|\mathrm{d}r\right)^{2}\mathrm{d}s\nonumber\\
	  &\leq 2M^{2}_{\alpha,\alpha}
	  \sup_{t\leq \tau\leq T}\int_{\tau}^{T}(s-\tau)^{\alpha-1}\textbf{E}\left(\int_{s}^{T}(r-s)^{\alpha-1}\mathrm{d}r \int_{s}^{T}(r-s)^{\alpha-1}\|f(r)\|^{2}\mathrm{d}r\right) \mathrm{d}s\nonumber\\
	&\leq 2M^{2}_{\alpha,\alpha}\frac{(T-t)^{\alpha}}{\alpha}  \sup_{t\leq \tau\leq T}\int_{\tau}^{T}(s-\tau)^{\alpha-1}\textbf{E}\int_{s}^{T}(r-s)^{\alpha-1}\|f(r)\|^{2}\mathrm{d}r\mathrm{d}s\nonumber\\
	&= 2M^{2}_{\alpha,\alpha}\frac{(T-t)^{\alpha}}{\alpha}  \sup_{t\leq \tau\leq T}\int_{\tau}^{T}(s-\tau)^{\alpha-1}\mathrm{d}s\|f\|^{2}_{2,\alpha,t}\nonumber\\
	&= 2M^{2}_{\alpha,\alpha}\frac{(T-t)^{2\alpha}}{\alpha^{2}}\|f\|^{2}_{2,\alpha,t}\leq 2M^{2}_{\alpha,\alpha}\frac{T^{2\alpha}}{\alpha^{2}}\|f\|^{2}_{2,\alpha,t}.
	 \end{align}
	 Substituting \eqref{I1} and \eqref{I2} into \eqref{xx}, and \eqref{fnorm} imply that
	 \begin{align}\label{6}
	 \|x\|^{2}_{2,\alpha,t}\leq 2M^{2}_{\alpha}\textbf{E}\|\left\lbrace \xi \mid \mathcal{F}_{\cdot} \right\rbrace\|^{2}_{2,\alpha,t}+2M^{2}_{\alpha,\alpha}\frac{T^{2\alpha}}{\alpha^{2}}\|f\|^{2}_{2,\alpha,t}.
	 \end{align}
Next we estimate $\tilde{y}$ by H\"{o}lder's inequality. Therefore, we get that
\begin{align*}
\|\tilde{y}(s,u)\|^{2}&\leq 2\|E_{\alpha}(A(T-s)^{\alpha}L(u))\|^{2}+2\norm{\int_{u}^{T}(r-s)^{\alpha-1}E_{\alpha,\alpha}(A(r-s)^{\alpha})K(r,u)\mathrm{d}r}^{2}\\
&\leq  2M^{2}_{\alpha}\|L(u)\|^{2}+2M^{2}_{\alpha,\alpha}\left(\frac{(T-s)^{2\alpha-1}}{2\alpha-1}-\frac{(u-s)^{2\alpha-1}}{2\alpha-1} \right) \int_{u}^{T}\|K(r,u)\|^{2}\mathrm{d}r\\
&\leq  2M^{2}_{\alpha}\|L(u)\|^{2}+2M^{2}_{\alpha,\alpha}\frac{(T+u)^{2\alpha-1}}{2\alpha-1} \int_{u}^{T}\|K(r,u)\|^{2}\mathrm{d}r\\
&\leq 2M^{2}_{\alpha}\|L(u)\|^{2}+2M^{2}_{\alpha,\alpha}\frac{(2T)^{2\alpha-1}}{2\alpha-1} \int_{u}^{T}\|K(r,u)\|^{2}\mathrm{d}r.
\end{align*}
Taking double integral of above inequality and applying stochastic Fubini's theorem twice yield that 
\allowdisplaybreaks
\begin{align*}
&\sup_{t \leq \tau \leq T}\textbf{E} \int_{\tau}^{T}(s-\tau)^{\alpha-1}\int_{s}^{T}\|\tilde{y}(s,u)\|^{2}\mathrm{d}u\mathrm{d}s \leq 2M^{2}_{\alpha}\sup_{t \leq \tau \leq T}\textbf{E}\int_{\tau}^{T}(s-\tau)^{\alpha-1}\int_{s}^{T}\|L(u)\|^{2}\mathrm{d}u\mathrm{d}s\nonumber\\
&+2M^{2}_{\alpha,\alpha}\frac{(2T)^{2\alpha-1}}{2\alpha-1} \sup_{t \leq \tau \leq T}\textbf{E}\int_{\tau}^{T}(s-\tau)^{\alpha-1}\int_{s}^{T}\int_{u}^{T}\|K(r,u)\|^{2}\mathrm{d}r\mathrm{d}u\mathrm{d}s\nonumber\\
&\leq 8M^{2}_{\alpha}\sup_{t\leq \tau\leq T}\int_{\tau}^{T}(s-\tau)^{2}\textbf{E}\|\xi\|^{2}\mathrm{d}s\\
&+2M^{2}_{\alpha,\alpha}\frac{(2T)^{2\alpha-1}}{2\alpha-1} \sup_{t \leq \tau \leq T}\textbf{E}\int_{\tau}^{T}(s-\tau)^{\alpha-1}\int_{s}^{T}\int_{s}^{r}\|K(r,u)\|^{2}\mathrm{d}u\mathrm{d}r\mathrm{d}s\nonumber\\
&\leq 8M^{2}_{\alpha}\frac{(T-t)^{\alpha}}{\alpha}\textbf{E}\|\xi\|^{2}+8M^{2}_{\alpha,\alpha}\frac{(2T)^{2\alpha-1}}{2\alpha-1}\sup_{t \leq \tau \leq T}\textbf{E}\int_{\tau}^{T}(s-\tau)^{\alpha-1}\int_{s}^{T}\|f(r)\|^{2}\mathrm{d}r \mathrm{d}s\nonumber\\
&\leq 8M^{2}_{\alpha}\frac{(T-t)^{\alpha}}{\alpha}\textbf{E}\|\xi\|^{2}+8M^{2}_{\alpha,\alpha}\frac{(2T)^{2\alpha-1}}{2\alpha-1}\sup_{t \leq \tau \leq T}\textbf{E}\int_{\tau}^{T}\int_{\tau}^{r}(s-\tau)^{\alpha-1}\|f(r)\|^{2}\mathrm{d}s \mathrm{d}r\nonumber\\
&\leq 8M^{2}_{\alpha}\frac{(T-t)^{\alpha}}{\alpha}\textbf{E}\|\xi\|^{2}+8M^{2}_{\alpha,\alpha}\frac{(2T)^{2\alpha-1}}{2\alpha-1}\sup_{t \leq \tau \leq T}\textbf{E}\int_{\tau}^{T}\int_{\tau}^{r}(s-\tau)^{\alpha-1}\mathrm{d}s\|f(r)\|^{2} \mathrm{d}r\\
&\leq 8M^{2}_{\alpha}\frac{(T-t)^{\alpha}}{\alpha}\textbf{E}\|\xi\|^{2}+8M^{2}_{\alpha,\alpha}\frac{(2T)^{2\alpha-1}(T-t)}{\alpha(2\alpha-1)} \sup_{t \leq \tau \leq T}\textbf{E}\int_{\tau}^{T}(r-\tau)^{\alpha-1}\|f(r)\|^{2}\mathrm{d}r\\
&\leq 8M^{2}_{\alpha}\frac{T^{\alpha}}{\alpha}\textbf{E}\|\xi\|^{2}+8M^{2}_{\alpha,\alpha}\frac{(2T)^{2\alpha-1}T}{\alpha(2\alpha-1)} \sup_{t \leq \tau \leq T}\textbf{E}\int_{\tau}^{T}(r-\tau)^{\alpha-1}\|f(r)\|^{2}\mathrm{d}r\\
&= 8M^{2}_{\alpha}\frac{T^{\alpha}}{\alpha}\textbf{E}\|\xi\|^{2}+4M^{2}_{\alpha,\alpha}\frac{(2T)^{2\alpha}}{\alpha(2\alpha-1)} \|f\|^{2}_{2,\alpha,t}.
\end{align*}
By definition of the norm \eqref{yt}, we obtain that
\begin{align}\label{7}
|\|\tilde{y}\||^{2}_{2,\alpha,t}\leq 8M^{2}_{\alpha}\frac{T^{\alpha}}{\alpha}\textbf{E}\|\xi\|^{2}+4M^{2}_{\alpha,\alpha}\frac{(2T)^{2\alpha}}{\alpha(2\alpha-1)} \|f\|^{2}_{2,\alpha,t}.
\end{align}
Since we have $y(t,s)=\tilde{y}(t,s)-g(t,s)$, by definition of the weighted norm \eqref{yt} we also have
\begin{equation}\label{norm}
|\|y\||^{2}_{2,\alpha,t}\leq 2|\|\tilde{y}\||^{2}_{2,\alpha,t}+2|\|g\||^{2}_{2,\alpha,t}.
\end{equation}
By taking into consideration \eqref{norm} for any $(x,y)$ the associated mild solution of \eqref{3} satisfies the following estimate:
\begin{align*}
\|x\|^{2}_{2,\alpha,t}+|\|y\||^{2}_{2,\alpha,t}&\leq \|x\|^{2}_{2,\alpha,t}+2|\|\tilde{y}\||^{2}_{2,\alpha,t}+2|\|g\||^{2}_{2,\alpha,t}\\
&\leq 2M^{2}_{\alpha}\textbf{E}\|\left\lbrace \xi \mid \mathcal{F}_{\cdot} \right\rbrace\|^{2}_{2,\alpha,t} +16M^{2}_{\alpha}\frac{T^{\alpha}}{\alpha}\textbf{E}\|\xi\|^{2}\\
&+2M^{2}_{\alpha,\alpha}\left(\frac{T^{2\alpha}}{\alpha^{2}}+ \frac{4(2T)^{2\alpha}}{\alpha(2\alpha-1)} \right) \|f\|^{2}_{2,\alpha,t}+2|\|g\||^{2}_{2,\alpha,t}.
\end{align*}
Therefore, the proof of Lemma \ref{lem1} is complete.
\end{proof}

\begin{theorem}\label{thm1}
	Let $\alpha\in (\frac{1}{2},1)$, $c>0$, $T>0$, let $M_{\alpha,\alpha}$ be defined as in \eqref{Mm}, and let $\mathfrak{L}(\alpha,c,T,M_{\alpha,\alpha})=\max\Big(2cT,2cM^{2}_{\alpha,\alpha}\left(\frac{T^{2\alpha}}{\alpha^{2}}+ \frac{4(2T)^{2\alpha}}{\alpha(2\alpha-1)} \right),2cM^{2}_{\alpha,\alpha}\left(\frac{T^{2\alpha}}{\alpha^{2}}+ \frac{4(2T)^{2\alpha}}{\alpha(2\alpha-1)} \right)T\Big)$. Assume that
	\begin{equation}\label{cond}
\mathfrak{L}(\alpha,c,T,M_{\alpha,\alpha})<1,
	\end{equation}
	then the system \eqref{fbsde} has a unique solution $(x,y) \in M[0,T]$ under Assumptions \ref{A0} and \ref{A1}.
\end{theorem}
\begin{proof}
	For any fixed $(\bar{x},\bar{y}) \in M[0,T]$, for $f(\cdot, \bar{x}(\cdot),\bar{y}(\cdot,\cdot)) \in L^{2}_{\mathcal{F}}([0,T],\mathbb{R}^{n} )$, and $g(\cdot,\cdot,\bar{x}(\cdot))\in L^{2}_{\mathcal{F}}(\mathcal{D},\mathbb{R}^{n\times m} )$, by Lemma \ref{lem1} we obtain that the following equation has a unique solution in $M[0,T]$:
	\begin{align}\label{8}
	x(t)=E_{\alpha}(A(T-t)^{\alpha})\xi &+\int_{t}^{T}(s-t)^{\alpha-1}E_{\alpha,\alpha}(A(s-t)^{\alpha})f(s,\bar{x}(s),\bar{y}(t,s))\mathrm{d}s\nonumber\\
	&+\int_{t}^{T}(s-t)^{\alpha-1}E_{\alpha,\alpha}(A(s-t)^{\alpha})\left[g(t,s, \bar{x}(s))+y(t,s) \right] \mathrm{d}w(s), \qquad \mathbb{P}\text{-a.s}.
	\end{align}
	Thus, the operator $\Psi: M[0,T] \to M[0,T]$ is defined by
	\begin{equation*}
	\Psi(\bar{x},\bar{y})=(x,y),
	\end{equation*}
	where $(x,y)$ is a solution of \eqref{8}, is well-defined.
	%%%%%%%%%%%%%%%%%%%%
	Now we prove the contractivity of the operator  $\Psi$. Therefore, Lemma \ref{lem1} and H\"{o}lder's inequality imply that 
	\allowdisplaybreaks
	\begin{align*}
	&\|\Psi(\bar{x},\bar{y})-\Psi(\tilde{x},\tilde{y})\|^{2}_{0}\\
	&\leq 2M^{2}_{\alpha,\alpha}\left(\frac{T^{2\alpha}}{\alpha^{2}}+ \frac{4(2T)^{2\alpha}}{\alpha(2\alpha-1)} \right)\sup_{0 \leq \tau \leq T}\textbf{E}\int_{\tau}^{T}(s-\tau)^{\alpha-1}\int_{s}^{T} \|f(s,\bar{x}(s),\bar{y}(t,s))-f(s,\tilde{x}(s),\tilde{y}(t,s))\|^{2}\mathrm{d}t\mathrm{d}s\\
	&+2\sup_{0 \leq \tau \leq T}\textbf{E}\int_{\tau}^{T}(s-\tau)^{\alpha-1}\int_{s}^{T}\|g(t,s,\bar{x}(s))-g(t,s,\tilde{x}(s))\|^{2}\mathrm{d}t\mathrm{d}s\\
	&\leq 2cM^{2}_{\alpha,\alpha}\left(\frac{T^{2\alpha}}{\alpha^{2}}+ \frac{4(2T)^{2\alpha}}{\alpha(2\alpha-1)} \right)\sup_{0 \leq \tau \leq T}\textbf{E}\int_{\tau}^{T}(s-\tau)^{\alpha-1}\int_{s}^{T}\left(\|\bar{x}(s)-\tilde{x}(s)\|^2+\|\bar{y}(t,s)-\tilde{y}(t,s)\|^{2}\right)\mathrm{d}t \mathrm{d}s\\
	&+2c\sup_{0 \leq \tau \leq T}\textbf{E}\int_{\tau}^{T}(s-\tau)^{\alpha-1}\int_{s}^{T}\|\bar{x}(s)-\tilde{x}(s)\|^{2}\mathrm{d}t\mathrm{d}s\\
	&\leq 2cM^{2}_{\alpha,\alpha}\left(\frac{T^{2\alpha}}{\alpha^{2}}+ \frac{4(2T)^{2\alpha}}{\alpha(2\alpha-1)} \right)\sup_{0 \leq \tau \leq T}\textbf{E}\int_{\tau}^{T}(s-\tau)^{\alpha-1}\int_{s}^{T}\|\bar{x}(s)-\tilde{x}(s)\|^2\mathrm{d}t 
	\mathrm{d}s\\
	&+2c\sup_{0 \leq \tau \leq T}\textbf{E}\int_{\tau}^{T}(s-\tau)^{\alpha-1}\int_{s}^{T}\|\bar{x}(s)-\tilde{x}(s)\|^{2}\mathrm{d}t\mathrm{d}s\\
	 &+2cM^{2}_{\alpha,\alpha}\left(\frac{T^{2\alpha}}{\alpha^{2}}+ \frac{4(2T)^{2\alpha}}{\alpha(2\alpha-1)} \right)\sup_{0 \leq \tau \leq T}\textbf{E}\int_{\tau}^{T}(s-\tau)^{\alpha-1}\int_{s}^{T}\|\bar{y}(t,s)-\tilde{y}(t,s)\|^{2}\mathrm{d}t \mathrm{d}s\\
	&\leq 2cM^{2}_{\alpha,\alpha}\left(\frac{T^{2\alpha}}{\alpha^{2}}+ \frac{4(2T)^{2\alpha}}{\alpha(2\alpha-1)} \right)T\sup_{0 \leq \tau \leq T}\textbf{E}\int_{\tau}^{T}(s-\tau)^{\alpha-1}\|\bar{x}(s)-\tilde{x}(s)\|^2\mathrm{d}s\\
	&+2cT\sup_{0 \leq \tau \leq T}\textbf{E}\int_{\tau}^{T}(s-\tau)^{\alpha-1}\|\bar{x}(s)-\tilde{x}(s)\|^{2}\mathrm{d}s\\
	&+2cM^{2}_{\alpha,\alpha}\left(\frac{T^{2\alpha}}{\alpha^{2}}+ \frac{4(2T)^{2\alpha}}{\alpha(2\alpha-1)} \right)\sup_{0 \leq \tau \leq T}\textbf{E}\int_{\tau}^{T}(s-\tau)^{\alpha-1}\int_{s}^{T}\|\bar{y}(t,s)-\tilde{y}(t,s)\|^{2}\mathrm{d}t \mathrm{d}s\\
	&=\mathfrak{L}(\alpha,c,T,M_{\alpha,\alpha})\left(  \|\bar{x}-\tilde{x}\|^{2}_{2,\alpha,0}+|\|\bar{y}-\tilde{y}\||_{2,\alpha,0}^{2}\right).
	\end{align*}
This together with \eqref{cond} assures a contractivity of the operator $\Psi$ on $M[0,T]$, which in turn implies the existence and uniqueness of the solution to \eqref{8}.
\end{proof}

\section{A coincidence between the notion of integral equation and mild solution}\label{coincidence}
In this section, we are going to prove the coincidence between the notion of stochastic Volterra integral equation and mild solution by the following theorem. 
\begin{theorem}\label{theorem1}
	The unique mild solution of\eqref{fbsde} with terminal value $\xi$ on $[0,T]$ is given by 
	\begin{align*}
	\psi(t,\xi)=E_{\alpha}(A(T-t)^{\alpha})\xi &+\int_{t}^{T}(s-t)^{\alpha-1}E_{\alpha,\alpha}(A(s-t)^{\alpha})f(s,\psi(s,\xi),y(t,s))\mathrm{d}s\\
	&+\int_{t}^{T}(s-t)^{\alpha-1}E_{\alpha,\alpha}(A(s-t)^{\alpha})\left[g(t,s,\psi(s,\xi))+y(t,s) \right] \mathrm{d}w(s).
	\end{align*}
\end{theorem}
\begin{proof}
	Before providing the proof of above theorem, we first need to show preparatory lemma and remark.
	In doing so, we present martingale representation theorem for any function $h \in L^{2}(\Omega, \mathcal{F}_{T}, \mathbb{R}^{m})$, there exists a unique adapted process $\Theta \in L^{2}(\Omega, \mathcal{F}_{T}, \mathbb{R}^{m})$ such that
	\begin{equation*}
	h=\textbf{E}h+\int_{0}^{T}\Theta(s)\mathrm{d}w(s).
	\end{equation*}
	It is clear that
	\begin{equation*}
	h= \sum_{i=1}^{m}h_{i}e_{i}= \sum_{i=1}^{m}\Big(\textbf{E}h_{i}+\int_{0}^{T}\theta_{i}(s)\mathrm{d}w(s)\Big)e_{i},
	\end{equation*}
	where
	\begin{equation*}
	h_{i}=\textbf{E}h_{i}+\int_{0}^{T}\theta_{i}(s)\mathrm{d}w(s), \quad h_{i} \in L^{2}(\Omega, \mathcal{F}_{T}, \mathbb{R}^{m}).
	\end{equation*}
	It is sufficient to show that 
	\begin{equation} \label{phi and psi}
	\psi(t,\xi)=\bar{\psi}(t,\xi).
	\end{equation}
	To show (\ref{phi and psi}), it is enough to prove that for any $h \in L^{2}(\Omega, \mathbb{F}_{T}, \mathbb{R}^{m})$,
	\begin{equation*}
	\textbf{E}\langle \psi(t,\xi),h \rangle = \textbf{E}\langle \bar{\psi}(t,\xi), h \rangle.
	\end{equation*}
	In other words, we have that
	\begin{equation*}
	\textbf{E}\langle \psi(t,\xi)-\bar{\psi}(t, \xi),h \rangle = \sum_{i=1}^{m}\textbf{E}\langle (\psi(t,\xi)-\bar{\psi}(t,\xi))h_{i}, e_{i} \rangle.
	\end{equation*}
	It follows that
	\begin{align*}
	|\textbf{E}\langle \psi(t,\xi)-\bar{\psi}(t, \xi),h \rangle|^{2} &\leq \|\sum_{i=1}^{m}\textbf{E}(\psi(t,\xi)-\bar{\psi}(t,\xi))h_{i}\|^{2}\\
	&\leq \sum_{i=1}^{m} 1^{2} \sum_{i=1}^{m}\|\textbf{E}(\psi(t,\xi)-\bar{\psi}(t,\xi))h_{i}\|^{2}. \nonumber \\
	&\leq m\sum_{i=1}^{m}\|\textbf{E}(\psi(t,\xi)-\bar{\psi}(t,\xi))h_{i}\|^{2}. \nonumber
	\end{align*}
	Before estimating $|\textbf{E}\langle \psi(t,\xi)-\bar{\psi}(t, \xi),h \rangle|$, we define the following functions:
	\begin{align*}
	&\chi_{i}(t)= \textbf{E}\psi(t,\xi)h_{i}, \quad \kappa_{i}(t)= \textbf{E}f(t,\psi(t,\xi),y(t,s))h_{i}, \\
	&\tilde{\chi}_{i}(t)= \textbf{E}\bar{\psi}(t,\xi)h_{i}, \quad
	\tilde{\kappa}_{i}(t)= \textbf{E}f(t,\bar{\psi}(t,\xi),y(t,s))h_{i}.
	\end{align*}
	\begin{remark} \label{remark1}
		Since $\psi(t,\xi), \bar{\psi}(t,\xi) \in M[0,T]$ , the functions $\chi_{i}, \kappa_{i}, \tilde{\chi}_{i}, \tilde{\kappa}_{i}$ are measurable and bounded  on $[t,T]$. 
	\end{remark}	
	\begin{lemma}
		For all $t \in [0,T]$ and $c \in \mathbb{R}^{m}$, the following statements hold:	
		\begin{align} \label{chi}
	\chi_{i}(t)=cE_{\alpha}(A(T-t)^{\alpha})\textbf{E}\xi &+\int_{t}^{T}(s-t)^{\alpha-1}E_{\alpha,\alpha}(A(s-t)^{\alpha})\kappa_{i}(s)\mathrm{d}s\nonumber\\
	&+\int_{t}^{T}(s-t)^{\alpha-1}E_{\alpha,\alpha}(A(s-t)^{\alpha})\textbf{E}\left\lbrace \theta_{i}(s)\left[g(t,s,\psi(s,\xi))+y(t,s) \right]\right\rbrace  \mathrm{d}w(s)
		\end{align}
		and
		\begin{align} \label{tilde{chi}}
		\tilde{\chi}_{i}(t)=cE_{\alpha}(A(T-t)^{\alpha})\textbf{E}\xi &+\int_{t}^{T}(s-t)^{\alpha-1}E_{\alpha,\alpha}(A(s-t)^{\alpha})\tilde{\kappa}_{i}(s)\mathrm{d}s\nonumber\\
		&+\int_{t}^{T}(s-t)^{\alpha-1}E_{\alpha,\alpha}(A(s-t)^{\alpha})\textbf{E}\left\lbrace \theta_{i}(s)\left[g(t,s,\bar{\psi}(s,\xi))+y(t,s) \right]\right\rbrace  \mathrm{d}w(s), \quad \mathbb{P}\text{-a.s}.
		\end{align}
	\end{lemma}
	
	\begin{proof}
		Since $\psi(t,\xi)$ is a solution of \eqref{fbsde}, it follows that 
		\begin{align}\label{Phi}
		\psi(t,\xi)&=\xi-\frac{1}{\Gamma(\alpha)}\int_{t}^{T}A(s-t)^{\alpha-1}\psi(s,\xi)\mathrm{d}s\nonumber\\
		&+\frac{1}{\Gamma(\alpha)}\int_{t}^{T}(s-t)^{\alpha-1}f(s,\psi(s,\xi),y(t,s))\mathrm{d}s\nonumber\\
		&+\frac{1}{\Gamma(\alpha)}\int_{t}^{T}(s-t)^{\alpha-1}\left[g(t,s,\psi(s,\xi))+y(t,s) \right] \mathrm{d}w(s).
		\end{align}
		By taking product of both sides of \eqref{Phi} with $h_{i}$  and then taking expectation of both sides, and finally by the property of stochastic integrals, we obtain that
		\begin{align*}
		\chi_{i}(t)&=c\textbf{E}\xi-\frac{1}{\Gamma(\alpha)}\int_{t}^{T}A(s-t)^{\alpha-1}\chi_{i}(s)\mathrm{d}s\\
		&+\frac{1}{\Gamma(\alpha)}\int_{t}^{T}(s-t)^{\alpha-1}\kappa_{i}(s)\mathrm{d}s\\
		&+\frac{1}{\Gamma(\alpha)}\int_{t}^{T}(s-t)^{\alpha-1}\textbf{E}\left\lbrace \theta_{i}(s)\left[g(t,s,\psi(s,\xi))+y(t,s) \right]\right\rbrace\mathrm{d}s\\
		&=c\textbf{E}\xi-\frac{1}{\Gamma(\alpha)}\int_{t}^{T}(s-t)^{\alpha-1}\Big(A\chi_{i}(s)-\kappa_{i}(s)-\textbf{E}\left\lbrace \theta_{i}(s)\left[g(t,s,\psi(s,\xi))+y(t,s) \right]\right\rbrace \Big)\mathrm{d}s.
		\end{align*}
		Therefore,  $\chi_{i}(t)$ is a solution of the following fractional stochastic backward differential equation:
		\begin{equation}
		\prescript{C}{t}{D^{\alpha}_{T}}x(t)= Ax(t)-\kappa_{i}(t) -\textbf{E}\theta_{i}(t)\left\lbrace \left[g(t,s,\psi(t,\xi))+y(t,s) \right]\right\rbrace , \quad x(T)= c\textbf{E}\xi.
		\end{equation}
		Then, by means of Remark \ref{remark1},  \eqref{chi} is proved. 
		Next, we define
		\begin{align} \label{Psi}	
		\bar{\psi}(t,\xi)=E_{\alpha}(A(T-t)^{\alpha})\xi &+\int_{t}^{T}(s-t)^{\alpha-1}E_{\alpha,\alpha}(A(s-t)^{\alpha})f(s,\bar{\psi}(s,\xi),y(t,s))\mathrm{d}s\nonumber\\
		&+\int_{t}^{T}(s-t)^{\alpha-1}E_{\alpha,\alpha}(A(s-t)^{\alpha})\left[g(t,s,\bar{\psi}(s,\xi))+y(t,s) \right] \mathrm{d}w(s), \qquad \mathbb{P}\text{-a.s.}
		\end{align}
		Similarly, \eqref{tilde{chi}} is proved as follows:
		\begin{align*} 	
		\tilde{\chi}_{i}(t)=cE_{\alpha}(A(T-t)^{\alpha})\textbf{E}\xi &+\int_{t}^{T}(s-t)^{\alpha-1}E_{\alpha,\alpha}(A(s-t)^{\alpha})\tilde{\kappa}_{i}(s)\mathrm{d}s\nonumber\\
		&+\int_{t}^{T}(s-t)^{\alpha-1}E_{\alpha,\alpha}(A(s-t)^{\alpha})\textbf{E}\left\lbrace \theta_{i}(s)\left[g(t,s,\bar{\psi}(s,\xi))+y(t,s) \right]\right\rbrace \mathrm{d}w(s), \qquad \mathbb{P}\text{-a.s}.
		\end{align*}
		Therefore, $\tilde{\chi}_{i}(t)$ is a solution of the following fractional stochastic backward differential equations:
		\begin{equation*}
		\prescript{C}{t}{D^{\alpha}_{T}}x(t)= Ax(t)-\tilde{\kappa}_{i}(t) -\textbf{E}\left\lbrace \theta_{i}(t)\left[g(t,s,\bar{\psi}(t,\xi))+y(t,s) \right]\right\rbrace 
		, \quad x(T)= cE_{\alpha}(A(T-t)^{\alpha})\textbf{E}\xi.
		\end{equation*}
	\end{proof}
	\begin{remark} \label{remark2}
		For any $h \in L^{2}(\Omega, \mathcal{F}_{T}, \mathbb{R}^{m})$, we have	
		
		\begin{align} \label{remproof}
		|\textbf{E}\langle \psi(t,\xi)-\bar{\psi}(t,\xi),h\rangle |^{2}
		&=m \sum_{i=1}^{m}\|\textbf{E}(\psi(t,\xi)-\bar{\psi}(t, \xi))h_{i}\|^{2} \nonumber\\
		&\leq 4mcM^{2}_{\alpha,\alpha}\frac{(T-t)^{2\alpha-1}}{2\alpha-1}\int_{t}^{T}\textbf{E}\|\psi(s,\xi)-\bar{\psi}(s,\xi)\|^{2}\mathrm{d}s  \textbf{E}\|h\|^{2}\nonumber\\
		&+4mcM^{2}_{\alpha,\alpha} \int_{t}^{T}(s-t)^{2\alpha-2}\textbf{E}\|\psi(s,\xi)-\bar{\psi}(s,\xi)\|^{2}\mathrm{d}s  \textbf{E}\|h\|^{2}.
		\end{align} 
	\end{remark}
	
	\begin{proof}
		To prove Remark \ref{remark2}, we start with the following inequality:
		\allowdisplaybreaks
		\begin{align} \label{X-Y}
		|\langle \psi(t,\xi)-\bar{\psi}(t,\xi),h \rangle|^{2}
		&=|\sum_{i=1}^{m}\textbf{E}\langle  \psi(t,\xi)-\bar{\psi}(t, \xi), h_{i} \rangle|^{2} \\ \nonumber
		&\leq \sum_{i=1}^{m} 1^{2} \sum_{i=1}^{m}\|\textbf{E}(\psi(t,\xi)-\bar{\psi}(t, \xi))h_{i}\|^{2} \\ \nonumber
		&\leq m\sum_{i=1}^{m}|\langle \psi(t,\xi)-\bar{\psi}(t,\xi),h_{i} \rangle |^{2}\\
		&\leq m\sum_{i=1}^{m}\|\textbf{E}( \psi(t,\xi)-\bar{\psi}(t,\xi))h_{i}\|^{2} \nonumber \\
		&\leq m\sum_{i=1}^{m}\|\chi_{i}(t)-\tilde{\chi}_{i}(t)\|^{2}. \nonumber
		\end{align}	
		First, we estimate $\|\chi_{i}(t)-\tilde{\chi}_{i}(t)\|$ as below:
		\begin{align*}
		\|\chi_{i}(t)-\tilde{\chi}_{i}(t)\|&\leq 
		M_{\alpha,\alpha}\int_{t}^{T} (s-t)^{\alpha-1}\|\kappa_{i}(s)-\tilde{\kappa}_{i}(s)\|\mathrm{d}s\nonumber\\
		&+mM_{\alpha,\alpha}\int_{t}^{T} (s-t)^{\alpha-1}\textbf{E}(\|\theta_{i}(s)\|\|\psi(s,\xi)-\bar{\psi}(s,\xi)\|)\mathrm{d}s.
		\end{align*}	
		Applying Cauchy-Schwartz inequality yields that
		\begin{align} \label{est}
		\|\chi_{i}(t)-\tilde{\chi}_{i}(t)\| &\leq M_{\alpha,\alpha}\sqrt{\frac{(T-t)^{2\alpha-1}}{2\alpha-1}}\Big( \int_{t}^{T}\|\kappa_{i}(s)-\tilde{\kappa}_{i}(s)\|^{2}\mathrm{d}s\Big)^{\frac{1}{2}}\nonumber \\
		&+mM_{\alpha,\alpha}\Big(\int_{t}^{T}\textbf{E} \|\theta_{i}(s)\|^{2}ds \Big)^{\frac{1}{2}} \Big(\int_{t}^{T}(s-t)^{2\alpha-2}\textbf{E}\|\psi(s,\xi)-\bar{\psi}(s,\xi)\|^{2} \mathrm{d}s \Big)^{\frac{1}{2}}. 
		\end{align}
		By definition of $\kappa$ and  $\tilde{\kappa}$ for all $s \in [t,T]$,
		\begin{align*}
		\|\kappa_{i}(s)-\tilde{\kappa}_{i}(s)\|^{2}&=\|\textbf{E}(f(s,\psi(s,\xi),y(t,s))-f(s,\bar{\psi}(s,\xi),y(t,s)))h_{i}\|^{2} \\
		&=\sum_{i=1}^{m}| \textbf{E} \langle f_{i}(s, \psi(s, \xi),y(t,s))-f_{i}(s, \bar{\psi}(s, \xi),y(t,s)), h_{i}\rangle|^{2}\\
		&\leq \sum_{i=1}^{m} \textbf{E}\|f_{i}(s, \psi(s, \xi),y(t,s))-f_{i}(s, \bar{\psi}(s, \xi),y(t,s))\|^{2} \textbf{E}\|h_{i}\|^{2}\\
		&= \textbf{E}\|f(s,\psi(s,\xi),y(t,s))-f(s,\bar{\psi}(s,\xi),y(t,s))\|^{2}\textbf{E}\|h_{i}\|^{2}\\
		&\leq c\textbf{E}\|\psi(s, \xi)-\bar{\psi}(s,\xi)\|^{2} \textbf{E}\| h_{i}\|^{2}.
		\end{align*}
		Then, we make use of above estimation in \eqref{est}, we obtain that
		\begin{align} 
		\|\chi_{i}(t)-\tilde{\chi}_{i}(t)\| &\leq mM_{\alpha,\alpha}\sqrt{\frac{(T-t)^{2\alpha-1}}{2\alpha-1}}\Big( \int_{t}^{T}\left( \textbf{E}\|\psi(s, \xi)-\bar{\psi}(s,\xi)\|^{2}\mathrm{d}s\right)^{\frac{1}{2}} \left( \textbf{E}\| h_{i}\|^{2}\right)^{\frac{1}{2}}\nonumber \\
		&+mM_{\alpha,\alpha}\Big(\int_{t}^{T}\textbf{E} \|\theta_{i}(s)\|^{2}ds \Big)^{\frac{1}{2}} \Big(\int_{t}^{T}(s-t)^{2\alpha-2}\textbf{E}\|\psi(s,\xi)-\bar{\psi}(s,\xi)\|^{2} \mathrm{d}s \Big)^{\frac{1}{2}}.
		\end{align}
		Now we take expectation of \eqref{X-Y} and plug \eqref{est} into \eqref{X-Y}, we attain desired result as follows:
		\begin{align*}
		|\textbf{E}\langle \psi(t,\xi)-\bar{\psi}(t,\xi),h\rangle |^{2} &=|\sum_{i=1}^{m}\textbf{E}\langle  \psi(t,\xi)-\bar{\psi}(t, \xi), h_{i} \rangle|^{2} \\
		&\leq m \sum_{i=1}^{m}\|\textbf{E}(\psi(t,\xi)-\bar{\psi}(t, \xi))h_{i}\|^{2} \\
		&\leq m \sum_{i=1}^{m}\|\textbf{E}(\psi(t,\xi)-\bar{\psi}(t, \xi))h_{i}\|^{2} \\
		&\leq 4mcM^{2}_{\alpha,\alpha}\frac{(T-t)^{2\alpha-1}}{2\alpha-1}\int_{t}^{T}\textbf{E}\|\psi(s,\xi)-\bar{\psi}(s,\xi)\|^{2}\mathrm{d}s  \textbf{E}\|h\|^{2}\\
		&+4cmM^{2}_{\alpha,\alpha} \int_{t}^{T}(s-t)^{2\alpha-2}\textbf{E}\|\psi(s,\xi)-\bar{\psi}(s,\xi)\|^{2}\mathrm{d}s  \textbf{E}\|h\|^{2},
		\end{align*}
		which completes the proof.
	\end{proof}
	
	\textit{Proof of Theorem \ref{theorem1}.} Let $T^{*}= \inf \left\lbrace t \in \left[0,T\right]; \psi(t,\xi) \neq \bar{\psi}(t,\xi)\right\rbrace$. Then it is sufficient to show that $T^{*}=T$. \\
	Suppose the contrary : $T^{*}< T$. Choose and fix an arbitrary $\delta >0$ satisfying the following expression:
	\begin{equation} \label{star}
	4mcM^{2}_{\alpha,\alpha}\frac{(T-t)^{2\alpha-1}}{2\alpha-1}\delta +4m cM^{2}_{\alpha,\alpha}\frac{\delta^{2\alpha-1}}{2\alpha-1} <1.
	\end{equation}
	To lead contradiction, we show that $\psi(t,\xi) = \bar{\psi}(t,\xi)$ for all $t \in [T^{*}-\delta,T^{*}]$. Using Ito's isometry, there exists a unique $h \in L^{2}$ such that  $\psi(t,\xi) - \bar{\psi}(t,\xi)= h$. Therefore, we have 
	\begin{equation*}
	\textbf{E}\|\psi(t,\xi) - \bar{\psi}(t,\xi)\|^{2}= \textbf{E}\|h\|^{2}.
	\end{equation*}
	Using Remark \eqref{remark2}, we attain 
	\begin{align*}
	\textbf{E}\|\psi(t,\xi) - \bar{\psi}(t,\xi)\|^{2}&\leq 4mcM^{2}_{\alpha,\alpha}\frac{(T-t)^{2\alpha-1}}{2\alpha-1} \int_{t}^{T^{*}}\textbf{E}\| \psi(s,\xi) - \bar{\psi}(s,\xi)\|^{2}\mathrm{d}s \\
	&+4mcM^{2}_{\alpha,\alpha} \int_{t}^{T^{*}}(s-t)^{2\alpha-2}\textbf{E}\| \psi(s,\xi) - \bar{\psi}(s,\xi)\|^{2}\mathrm{d}s.
	\end{align*}  
	As a consequence, we have that
	\begin{align*}
	\sup_{t \in [T^{*}-\delta,T^{*}]}\textbf{E}\|\psi(t,\xi) - \bar{\psi}(t,\xi)\|^{2} 
	&\leq \Big[4mcM^{2}_{\alpha,\alpha}\frac{(T-t)^{2\alpha-1}}{2\alpha-1}\delta+4mcM^{2}_{\alpha,\alpha}\frac{\delta^{2\alpha-1}}{2\alpha-1}\Big]\\
	&\times \sup_{t \in [T^{*}-\delta,T^{*}]}\textbf{E}\|\psi(t,\xi) - \bar{\psi}(t,\xi)\|^{2}.
	\end{align*}
	By selecting $\delta$ as in \eqref{star}, we have  $\sup_{t \in [T^{*}-\delta,T^{*}]}\textbf{E}\|\psi(t,\xi) - \bar{\psi}(t,\xi)\|^{2}=0$. This leads to a contradiction and the proof is thus complete.
\end{proof}
\section{Conclusions and future works}\label{open problems}
The main contributions of our work are described in detail below:
	\begin{itemize}
		\item  It is worthy mention that we first formulated a new problem in BSDE theory, namely the so-called Caputo fBSDE, which is an untreated topic in recent literature;
		\item we then introduced a new weighted norm in the square integrable measurable function space that is useful for proving a fundamental lemma and its well-posedness;
		\item  To derive an adapted pair of stochastic processes, we established a fundamental lemma which plays a crucial role in the theory of Caputo fBSDE;
		\item The main results in our article were to show global existence and uniqueness of an adapted solution to \eqref{fbsde} in finite dimensional setting with the help of fundamental lemma and a new weighted maximum norm in the square-integrable measurable space. The key point in the proof of this lemma was to apply the extended martingale representation theorem and some inequalities from stochastic calculus;
		\item  We derived a mild solution of fBSDE and proved the coincidence between the notion of integral equations and mild solutions using the martingale representation theorem.
	\end{itemize}
Since our results are sufficiently new in the theory of BSDEs, there are still open problems to discuss related to fractional stochastic control theory and risk sensitive control problems. If one can consider the same statement of our results shown above, an interesting problem in Caputo fBSDE appears for multi-dimensional case, for example $m \geq 2$. The problem of existence with the coefficient $f$ only continuous in $(x,y) $ becomes very hard. A similar problem reveals in situation of the quadratic growth condition. Successfully applied techniques in the one-dimensional case fail here due to the lack of comparison theorems. In general, the multi-dimensional case has many interesting applications. 

\section*{Data availability statement}
Data will be made available on reasonable request.

\end{document}